\documentclass[10pt,oneside]{amsart}

\usepackage[
paper=a4paper,
headsep=15pt,text={135mm,216mm},centering
]{geometry}

\usepackage[bookmarks]{hyperref}
\usepackage{amssymb,amsxtra}
\usepackage{graphicx,xcolor}
\usepackage{float,array,calc,booktabs,stackrel,url}
\usepackage{pb-diagram} 
\usepackage{enumitem}\setlist[enumerate,1]{font=\upshape}
\usepackage{mathtools}
\usepackage{tikz}
\usetikzlibrary{
  cd,
  calc,
  positioning,
  arrows,
  decorations.pathreplacing,
  decorations.markings,
}
\usepackage[mode=buildnew]{standalone}

\definecolor{todo-background-color}{gray}{0.95}
\usepackage[
  textwidth=1.1in,
  backgroundcolor=todo-background-color,
  bordercolor=black,
  linecolor=black,
  textsize=footnotesize
]{todonotes}

\iftrue
    
    \makeatletter
    \def\@settitle{%
      \vspace*{-10pt}
      \begin{flushleft}%
        \LARGE\bfseries
        \strut\@title\strut
      \end{flushleft}%
    }
    \def\@setauthors{%
      \begingroup
      \def\thanks{\protect\thanks@warning}%
      \trivlist
      \raggedright
      \large \@topsep27\p@\relax
      \advance\@topsep by -\baselineskip
    \item\relax
      \author@andify\authors
      \def\\{\protect\linebreak}%
      \authors
      \ifx\@empty\contribs
      \else
      ,\penalty-3 \space \@setcontribs
      \@closetoccontribs
      \fi
      \normalfont
      \endtrivlist
      \endgroup
    }
    \def\@setaddresses{\par
      \nobreak \begingroup
      \small\raggedright
      \def\author##1{\nobreak\addvspace\smallskipamount}%
      \def\\{\unskip, \ignorespaces}%
      \interlinepenalty\@M
      \def\address##1##2{\begingroup
        \par\addvspace\bigskipamount\noindent
        \@ifnotempty{##1}{(\ignorespaces##1\unskip) }%
        {\ignorespaces##2}\par\endgroup}%
      \def\curraddr##1##2{\begingroup
        \@ifnotempty{##2}{\nobreak\noindent\curraddrname
          \@ifnotempty{##1}{, \ignorespaces##1\unskip}\/:\space
          ##2\par}\endgroup}%
      \def\email##1##2{\begingroup
        \@ifnotempty{##2}{\nobreak\noindent E-mail address%
          \@ifnotempty{##1}{, \ignorespaces##1\unskip}\/:\space
          \ttfamily##2\par}\endgroup}%
      \def\urladdr##1##2{\begingroup
        \def~{\char`\~}%
        \@ifnotempty{##2}{\nobreak\noindent\urladdrname
          \@ifnotempty{##1}{, \ignorespaces##1\unskip}\/:\space
          \ttfamily##2\par}\endgroup}%
      \addresses
      \endgroup
      \global\let\addresses=\@empty
    }
    \def\@setabstracta{%
      \ifvoid\abstractbox
      \else
      \skip@17pt \advance\skip@-\lastskip
      \advance\skip@-\baselineskip \vskip\skip@
      \box\abstractbox
      \prevdepth\z@ 
      \vskip-28pt
      \fi
    }
    \renewenvironment{abstract}{%
      \ifx\maketitle\relax
      \ClassWarning{\@classname}{Abstract should precede
        \protect\maketitle\space in AMS document classes; reported}%
      \fi
      \global\setbox\abstractbox=\vtop \bgroup
      \normalfont\small
      \list{}{\labelwidth\z@
        \leftmargin0pc \rightmargin\leftmargin
        \listparindent\normalparindent \itemindent\z@
        \parsep\z@ \@plus\p@
        
      }%
    \item[\hskip\labelsep\bfseries\abstractname.]%
    }{%
      \endlist\egroup
      \ifx\@setabstract\relax \@setabstracta \fi
    }

    \def\ps@headings{\ps@empty
      \def\@evenhead{%
        \setTrue{runhead}%
        \normalfont\scriptsize
        \rlap{\thepage}\hfill
        \def\thanks{\protect\thanks@warning}%
        \leftmark{}{}}%
      \def\@oddhead{%
        \setTrue{runhead}%
        \normalfont\scriptsize
        \def\thanks{\protect\thanks@warning}%
        \rightmark{}{}\hfill \llap{\thepage}}%
      \let\@mkboth\markboth
    }\ps@headings

    \def\section{\@startsection{section}{1}%
      \z@{-1.4\linespacing\@plus-.5\linespacing}{.8\linespacing}%
      {\normalfont\bfseries\Large}}
    \def\subsection{\@startsection{subsection}{2}%
      \z@{-.8\linespacing\@plus-.3\linespacing}{.5\linespacing\@plus.2\linespacing}%
      {\normalfont\bfseries\large}}
    \def\subsubsection{\@startsection{subsubsection}{3}%
      \z@{.7\linespacing\@plus.2\linespacing}{-1.5ex}%
      {\normalfont\itshape}}
    \def\paragraph{\@startsection{paragraph}{4}%
      \z@{.7\linespacing\@plus.2\linespacing}{-1.5ex}%
      {\normalfont\itshape}}
    \def\@secnumfont{\bfseries}

    \renewcommand\contentsnamefont{\bfseries}
    \def\@starttoc#1#2{\begingroup
      \setTrue{#1}%
      \par\removelastskip\vskip\z@skip
      \@startsection{}\@M\z@{\linespacing\@plus\linespacing}%
      {.5\linespacing}{
        \contentsnamefont}{#2}%
      \ifx\contentsname#2%
      \else \addcontentsline{toc}{section}{#2}\fi
      \makeatletter
      \@input{\jobname.#1}%
      \if@filesw
      \@xp\newwrite\csname tf@#1\endcsname
      \immediate\@xp\openout\csname tf@#1\endcsname \jobname.#1\relax
      \fi
      \global\@nobreakfalse \endgroup
      \addvspace{32\p@\@plus14\p@}%
      \let\tableofcontents\rela\x
    }
    \def\contentsname{Contents}
    \def\l@section{\@tocline{2}{.5ex}{0mm}{5pc}{}}
    \def\l@subsection{\@tocline{2}{0pt}{2em}{5pc}{}}
    \makeatother

\fi

\def\to{\mathchoice{\longrightarrow}{\rightarrow}{\rightarrow}{\rightarrow}}
\makeatletter
\newcommand{\shortxra}[2][]{\ext@arrow 0359\rightarrowfill@{#1}{#2}}
\def\longrightarrowfill@{\arrowfill@\relbar\relbar\longrightarrow}
\newcommand{\longxra}[2][]{\ext@arrow 0359\longrightarrowfill@{#1}{#2}}
\renewcommand{\xrightarrow}[2][]{\mathchoice{\longxra[#1]{#2}}%
  {\shortxra[#1]{#2}}{\shortxra[#1]{#2}}{\shortxra[#1]{#2}}}
\makeatother

\makeatletter
\def\addtagsub#1{\let\oldtf=\tagform@\def\tagform@##1{\oldtf{##1}\hbox{$_{#1}$}}}
\makeatother

\makeatletter
\def\Nopagebreak{\@nobreaktrue\nopagebreak}
\makeatother

\makeatletter
\newtheoremstyle{theorem-giventitle}
        {}{}              
        {\itshape}                      
        {}                              
        {\bfseries}                     
        {.}                             
        {\thm@headsep}                             
        {\thmnote{\bfseries#3}}
\newtheoremstyle{theorem-givenlabel}
        {}{}              
        {\itshape}                      
        {}                              
        {\bfseries}                     
        {.}                             
        {\thm@headsep}                             
        {\thmname{#1}~\thmnumber{#3}\setcurrentlabel{#3}}
\newtheoremstyle{definition-giventitle}
        {}{}              
        {}                      
        {}                              
        {\bfseries}                     
        {.}                             
        {\thm@headsep}                             
        {\thmnote{\bfseries#3}}
\def\setcurrentlabel#1{\gdef\@currentlabel{#1}}
\makeatother

\newtheorem{theorem}{Theorem}[section]
\newtheorem{theoremalpha}{Theorem}

\newtheorem{corollary}[theorem]{Corollary}
\newtheorem{lemma}[theorem]{Lemma}

\newtheorem{conjecturealpha}[theoremalpha]{Conjecture}

\theoremstyle{definition}
\newtheorem{definition}[theorem]{Definition}

\newtheorem*{case2'}{Case 2$'$}

\theoremstyle{theorem-giventitle}
\newtheorem{theorem-named}{}
\theoremstyle{theorem-givenlabel}
\newtheorem{theorem-labeled}{Theorem}

\theoremstyle{definition-giventitle}
\newtheorem{definition-named}{}
\newtheorem{conjecture-named}{}
\newtheorem{case-named}{}

\numberwithin{equation}{section}

\def\Z{\mathbb{Z}}

\def\Q{\mathbb{Q}}

\def\cR{\mathcal{R}}

\def\tilde{\widetilde}

\DeclareMathOperator\Ker{Ker}

\def\Im{\operatorname{Im}}

\DeclareMathOperator\sign{sign}

\def\rank{\operatorname{rank}}

\def\NN{\mathcal{N}}

\def\cC{\mathcal{C}}

\def\cF{\mathcal{F}}
\def\cK{\mathcal{K}}
\def\cA{\mathcal{A}}

\def\K{\mathbb{K}}

\makeatletter

\begin{document}

\title[Primary decomposition]{Primary decomposition of knot concordance and von Neumann rho-invariants}

\author{Min Hoon Kim}
\address{
  Department of Mathematics\\
  POSTECH \\
  Pohang Gyeongbuk 37673\\
  Republic of Korea
}
\email{kminhoon@gmail.com}

\author{Se-Goo Kim}
\address{Department of Mathematics and Research Institute for Basic Sciences\\
Kyung Hee University\\
Seoul 02447\\
Republic of Korea
}
\email{sgkim@khu.ac.kr}

\author{Taehee Kim}
\address{Department of Mathematics\\
  Konkuk University\\
  Seoul 05029\\
  Republic of Korea
}
\email{tkim@konkuk.ac.kr}

\thanks{The second named author was supported by the Basic Science Research Program through the National Research Foundation of Korea (NRF) funded by the Ministry of Education (NRF-2018R1D1A1B07047860). The last named author was supported by Basic Science Research Program through the National Research Foundation of Korea(NRF) funded by the Ministry of Education (no.2018R1D1A1B07048361).}

\def\subjclassname{\textup{2010} Mathematics Subject Classification}
\expandafter\let\csname subjclassname@1991\endcsname=\subjclassname
\expandafter\let\csname subjclassname@2000\endcsname=\subjclassname
\subjclass{%
  57N13, 
  57M27, 
  57N70, 
  57M25
}

\begin{abstract} We address the primary decomposition of the knot concordance group in terms of the solvable filtration and higher-order von Neumann $\rho$-invariants by Cochran, Orr, and Teichner. We show that for a nonnegative integer $n$, if the connected sum of two $n$-solvable knots with coprime Alexander polynomials is slice, then each of the knots has vanishing von Neumann $\rho$-invariants of order $n$. This gives positive evidence for the conjecture that nonslice knots with coprime Alexander polynomials are not concordant. As an application, we show that if $K$ is one of Cochran-Orr-Teichner's knots which are the first examples of nonslice knots with vanishing Casson-Gordon invariants, then $K$ is not concordant to any knot with Alexander polynomial coprime to that of $K$. 

\end{abstract}

\maketitle

\section{Introduction}
Two knots $K$ and $J$ are \emph{concordant} if there exists a proper and locally flat embedding of an annulus into $S^3\times [0,1]$ which gives $K\times \{1\} \sqcup -J\times \{0\}$ on the boundary. A knot is \emph{slice} if it is concordant to the unknot. Concordance is an equivalence relation, and the concordance classes form an abelian group $\cC$, which is called \emph{the knot concordance group}, under connected sum. Although the group $\cC$ is not classified yet, its algebraic analogue, the algebraic concordance group, was classified by Levine \cite{Levine:1969-1, Levine:1969-2};  the algebraic concordance group consists of equivalence classes of Seifert forms, and Levine classified it using the primary decomposition of Seifert forms along Alexander polynomials. Therefore, it is a natural idea to try to classify or find structures of the knot concordance group using a similar primary decomposition along Alexander polynomials. For details and related discussions on primary decomposition we refer the reader to \cite{Cha:2019-1}. 

In this paper we address a conjecture which would play a key role in the primary decomposition of the knot concordance group:

\begin{conjecturealpha}\label{conjecture:topological-splitting}
	Let $K$ and $J$ be nonslice knots. If $K$ and $J$ have coprime Alexander polynomials, then $K$ and $J$ are not concordant. 
\end{conjecturealpha}

Put differently, it is conjectured that if the connected sum of two knots with coprime Alexander polynomials is slice, then both knots are slice. We remark that the smooth concordance version of Conjecture~\ref{conjecture:topological-splitting} requires an additional assumption that $K$ and $J$ are not smoothly concordant to a knot with trivial Alexander polynomial \cite[Appendix]{Cha:2019-1}. Note that a knot with trivial Alexander polynomial is (topologically) slice \cite{Freedman:1982-2,Freedman-Quinn:1990-1}. 

There has been positive evidence for Conjecture~\ref{conjecture:topological-splitting}. It was shown that if the connected sum of two knots with coprime Alexander polynomials is slice, then the Casson-Gordon invariant and the metabelian von Neumann $\rho$-invariant vanish for both knots \cite{Kim:2005-2,Kim-Kim:2008-1}, and the higher-order von Neumann $\rho$-invariant vanish for both knots under a certain splitting condition for higher-order Blanchfield linking forms \cite{Kim-Kim:2014-1}. 

In this paper, we give another positive evidence for Conjecture~\ref{conjecture:topological-splitting}. In 2003, Cochran, Orr, and Teichner \cite{Cochran-Orr-Teichner:2003-1} introduced \emph{the solvable filtration of $\cC$},
\[
0\subset \cdots \subset \cF_{n.5}\subset \cF_n \subset \cdots \subset \cF_{0.5}\subset \cF_0,\subset \cC ,
\]
and \emph{the von Neumann $\rho$-invariants of order $n$} for each integer $n\ge 0$ as obstructions for a knot in $\cF_n$ to being in $\cF_{n.5}$ \cite[Theorem~4.6]{Cochran-Orr-Teichner:2003-1}. For each half-integer $h\ge 0$, a knot in $\cF_h$ is called \emph{$h$-solvable}, and it is known that each $\cF_h$ is a subgroup of $\cC$. See Definitions~\ref{definition:F_n} and \ref{definition:vanishing-rho-invariant} for the definitions of $\cF_n$ and $\rho$-invariants of order $n$, respectively. We remark that a knot is $0.5$-solvable if and only if it is algebraically slice, and if a knot is $1.5$-solvable, then the Casson-Gordon invariant and metabelian $\rho$-invariant vanish for the knot. We also note that it is still unknown whether or not the transfinite intersection $\bigcap_n \cF_n$ is trivial. In \cite{Cochran-Orr-Teichner:2003-1, Cochran-Orr-Teichner:2004-1}, Cochran, Orr, and Teichner gave the first examples of nonslice knots with vanishing Casson-Gordon invariants by showing that there is a 2-solvable knot which does not have vanishing $\rho$-invariants of order 2, hence not 2.5-solvable. 

In this paper, we prove the following theorem.

\begin{theorem}\label{theorem:main}
Let $n\ge 0$ be an integer. Suppose two knots $K$ and $J$ have coprime Alexander polynomials. If $K$ is $n$-solvable and $K\# J$ is $n.5$-solvable (e.g. slice), then $K$ has vanishing $\rho$-invariants of order $n$. 
\end{theorem}

We note that Theorem~\ref{theorem:main} extends Theorem~1.1 in \cite{Kim-Kim:2014-1} by removing the extra condition that the higher-order Blanchfield forms and self-annihilating submodules split. 

The aforementioned knots of Cochran, Orr, and Teichner, which have Alexander polynomial $(t^2-3t+1)^2$, are 2-solvable and do not have vanishing $\rho$-invariants of order 2 \cite{Cochran-Orr-Teichner:2003-1, Cochran-Orr-Teichner:2004-1}. Therefore, we obtain the following corollary. 

\begin{corollary}\label{corollary:COT-knot-coprime}
If $K$ is one of Cochran-Orr-Teichner's knots in \cite{Cochran-Orr-Teichner:2003-1,Cochran-Orr-Teichner:2004-1}, then $K$ is not concordant to any knot with Alexander polynomial coprime to $\Delta_K(t)=(t^2-3t+1)^2$.
\end{corollary}

The \emph{concordance genus} of a knot $K$ is the minimum genus among all knots concordant to $K$. Using Corollary~\ref{corollary:COT-knot-coprime}, we can reprove Theorem~1.2 in \cite{Kim-Kim:2014-1}, which gave the first examples of knots with vanishing Casson-Gordon invariants which have concordance genus greater than 1. 

\begin{corollary}[{\cite[Theorem~1.2]{Kim-Kim:2014-1}}]\label{corollary:CPT-knot-concordance-genus}
Cochran-Orr-Teichner's knots in \cite{Cochran-Orr-Teichner:2003-1,Cochran-Orr-Teichner:2004-1}, which are 2-solvable, have concordance genus 2.
\end{corollary}
\begin{proof}
Suppose that a knot $K$ has genus 1 and is concordant to one of Cochran-Orr-Teichner's knots. Then $K$ is also 2-solvable, hence algebraically slice. Since $K$ has genus 1, it follows that $\Delta_K(t)=(mt-(m+1))((m+1)t-m)$ for some integer $m\ge 1$, which is coprime to $(t^2-3t+1)^2$. This contradicts Corollary~\ref{corollary:COT-knot-coprime}.
\end{proof}

In this paper, all manifolds are oriented, compact, and connected, and homology groups are understood with integer coefficients unless mentioned otherwise.

This paper is organized as follows. In Section~\ref{section:rho-invariants} we introduce the notions of $n$-solvable knots and $\rho$-invariants of order $n$. In Section~\ref{section:proof-of-main-theorem} we give a proof of Theorem~\ref{theorem:main}.


\section{The solvable filtration and the von Neumann $\rho$-nvariants of order $n$}\label{section:rho-invariants}
In this section, we review the results in \cite{Cochran-Orr-Teichner:2003-1} which will be needed for the proof of Theorem~\ref{theorem:main}.
\subsection{The solvable filtration}\label{subsection:solvable-filtration}
In \cite{Cochran-Orr-Teichner:2003-1}, Cochran, Orr, and Teichner introduced the \emph{solvable filtration} $\{\cF_n\}$ of the knot concordance group $\cC$, which is indexed by nonnegative half-integers, that is,
\[
0\subset \cdots\subset \cF_{n.5}\subset \cF_n\subset \cdots\subset \cF_{0.5}\subset \cF_0\subset \cC,
\]
where $\cF_n$ is the subgroup of $n$-solvable knots. The definition of an $n$-solvable knot follows. Throughout this paper $M(K)$ denotes the zero-framed surgery on a knot $K$ in $S^3$. For a group $G$, we let $G^{(0)}:=G$ and $G^{(n+1)}:=[G^{(n)},G^{(n)}]$ for $n\ge 0$.

\begin{definition}\label{definition:F_n}
 Let $n$ be a nonnegative integer. A closed 3-manifold is \emph{$n$-solvable via $W$} if there exists a spin 4-manifold $W$ with $\partial W=M$ which satisfies the following: let $\pi:=\pi_1(W)$ and $r:=\rank_\Q H_2(W)$.
\begin{enumerate}
	\item The homomorphism $H_1(M)\to H_1(W)$ induced from inclusion is an isomorphism.
	\item There exist elements $x_1,x_2,\ldots, x_r$ and $y_1,y_2,\ldots, y_r$ such that for the equivarient intersection form
	\[
	\lambda_n\colon H_2(W;\Z[\pi/\pi^{(n)}])\times H_2(W;\Z[\pi/\pi^{(n)}]) \to \Z[\pi/\pi^{(n)}],
	\]
	$\lambda_n(x_i,x_j)=0$ and $\lambda_n(x_i,y_j)=\delta_{ij}$ for all $i$ and $j$.
\end{enumerate}
The above 4-manifold $W$ is called an \emph{$n$-solution}.
In addition to (1) and (2), if there exist lifts $\tilde{x}_i$ of $x_i$ for $1\le i\le r$ in $H_2(W;\Z[\pi/\pi^{(n+1)}])$ such that $\lambda_{n+1}(\tilde{x}_i,\tilde{x}_j)=0$ for all $i$ and $j$, then $M$ is called \emph{$n.5$-solvable via $W$} and $W$ is called an \emph{$n.5$-solution}.  We say that a knot $K$ is \emph{$n$-solvable} (resp. \emph{$n.5$-solvable}) if $M(K)$ is $n$-solvable (resp. $n.5$-solvable).
\end{definition}

\subsection{The von Neumann $\rho$-invariant}\label{subsection:rho-invariant}
The von Neumann $\rho$-invariant was defined by Cheeger and Gromov \cite{Cheeger-Gromov:1985-1}, and introduced for the study on knot concordance by Cochran, Orr, and Teichner \cite{Cochran-Orr-Teichner:2003-1}. In this paper, we give a brief topological definition of the von Neumann $\rho$-invariant using $L^2$-signatures. For more details, we refer the reader to \cite{Cochran-Orr-Teichner:2003-1,Cochran-Teichner:2003-1,Cha-Orr:2009-01,Cha:2010-01}.

Let $M$ be a closed manifold, $G$ be a countable discreted group, and $\phi\colon \pi_1(M)\to G$ be a homomorphism. Suppose there exists a 4-manifold $W$ with $\partial W=M$ such that $\phi$ extends to $\pi_1(W)$. Then the von Neumann $\rho$-invariant of $(M,\phi)$ is the $L^2$-signature defect of $W$, that is, 
\[
\rho(M,\phi)=\bar{\sigma}_G(W):=\sign_G^{(2)}(W)-\sign(W)
\]
where $\sign(W)$ is the ordinary signature of $W$ and $\sign_G^{(2)}(W)$ is the $L^2$-signature of the intersection form
\[
H_2(W;\NN G)\times H_2(W;\NN G)\to \NN G
\]
where $\NN G$ is the group von Neumann algebra of $G$. 

Since every group embeds into an acyclic group and since $\rho(M,\phi)=\rho(M,i\circ \phi)$ for every injective homomorphism $i\colon G\hookrightarrow G'$, the von Neumann $\rho$-invariant $\rho(M,\phi)$ can be computed by assuming the existence of $W$ as above.

\subsection{The rationally universal solvable groups}\label{subsection:rationally-universal-solvable-group}
In this subsection we review the rationally universal solvable groups defined in \cite[Section~3]{Cochran-Orr-Teichner:2003-1} via which the sliceness obstructions in \cite[Theorem~4.6]{Cochran-Orr-Teichner:2003-1} are defined.

The \emph{0th rationally universal solvable group} is defined to be $\Gamma_0:=\Z$. Now suppose $\Gamma_{n-1}$ has been defined. Then, the \emph{$n$th rationally universal solvable group} $\Gamma_n$ is defined inductively as follows: let $\cR_{n-1}:=\Q\Gamma_{n-1}(\Q(\Gamma_{n-1}^{(1)})\setminus \{0\})^{-1}$ and $\cK_{n-1}:=\Q\Gamma_{n-1}(\Q\Gamma_{n-1}\setminus \{0\})^{-1}$. Then, we define $\Gamma_n:=(\cK_{n-1}/\cR_{n-1})\rtimes \Gamma_{n-1}$ where $\Gamma_{n-1}$ acts on $\cK_{n-1}/\cR_{n-1}$ via multiplication on the right. 

We remark that for each $n$, the (possibly noncommutative) ring $\cR_n$ is a left and right PID and $\cK_n$ is the (skew) quotient field of $\Z\Gamma_n$. To be precise, if we let $\Gamma_0=\langle t\rangle$ and $\K_n:=\Q\Gamma_n^{(1)}(\Q\Gamma_n^{(1)}\setminus \{0\})^{-1}$, the (skew) quotient field of $\Q\Gamma_n^{(1)}$, then $\cR_n\cong \K_n[t^{\pm 1}]$ and $\cK_n\cong \K_n(t)$ (see \cite[Corollary~3.3]{Cochran-Orr-Teichner:2003-1}). For instance, we have $\K_0\cong \Q$, and hence $\cR_0\cong \Q[t^{\pm 1}]$ and $\cK_0\cong \Q(t)$. 

\subsection{Representations to $\Gamma_n$}\label{subsection:representations-to-Gamma}
In this subsection, we review the work in \cite[Section~3]{Cochran-Orr-Teichner:2003-1} regarding how to construct representations to $\Gamma_n$ inductively. 

Let $K$ be a knot and $M:=M(K)$. Then we have the abelianization $\epsilon\colon \pi_1(M)\to \Z=\Gamma_0$ and the rational Alexander module $H_1(M;\cR_0)$. Define \emph{the 0th order Alexander module} $\cA_0:=H_1(M;\cR_0)$, which is in fact the rational Alexander module for $M$. We also define \emph{the 0th order Blanchfield form} to be the rational Blanchfield form  
\[
B\ell_0\colon \cA_0\times \cA_0\to \cK_n/\cR_0.
\]

 Choose an element $x_0\in \cA_0$. Note that there is a canonical projection $p_0\colon \Gamma_1\to \Gamma_0$ which maps $(a,b)\in (\cK_0/\cR_0)\rtimes \Gamma_0$ to $b\in \Gamma_0$. By \cite[Theorem~3.5]{Cochran-Orr-Teichner:2003-1} there exists a homomorphism $\phi_{x_0}\colon \pi_1(M)\to \Gamma_1$ such that 
 \[
 \phi_{x_0}(y)= (B\ell_0(x_0,y), 0)\in (\cK_0/\cR_0)\rtimes \Gamma_0=\Gamma_1
 \]
 for each $y\in \cA_0$ and $p_0\circ\phi_{x_0}=\epsilon$. Also, we obtain a coefficient system $\Z\pi_1(M)\xrightarrow{\phi_{x_0}} \Z\Gamma_1\to \cR_1$ and the corresponding \emph{first order Alexander module} $\cA_1=\cA_1(x_0):=H_1(M;\cR_1)$. Then, Cochran, Orr, and Teichner showed that $\cA_1$ is a (right) $\cR_1$-torsion module and defined \emph{the first order Blanchfield form} 
\[
B\ell_1\colon \cA_1\times \cA_1\to \cK_1/\cR_1.
\]

We iterate the above process. Suppose the homomorphism 
\[
\phi_{x_0,x_1,\ldots, x_{n-1}}\colon \pi_1(M)\to \Gamma_n
\] 
and the $n$th order Alexander module $\cA_n=\cA_n(x_0,x_1,\ldots,x_{n-1})=H_1(M;\cR_n)$ have been defined. Also suppose the $n$th order Blanchfield form 
\[
B\ell_n\colon \cA_n\times \cA_n\to \cK_n/\cR_n
\]
have been defined. Choose an element $x_n\in \cA_n$. Let $p_n\colon \Gamma_{n+1}\to \Gamma_n$ be the canonical projection. Then, by \cite[Theorem~3.5]{Cochran-Orr-Teichner:2003-1} there exists a homomorphism 
\[
\phi_{x_0,x_1,\ldots,x_n}\colon \pi_1(M)\to \Gamma_{n+1}
\] 
such that $\phi_{x_0,x_1,\ldots,x_n}(y)= (B\ell_n(x_n,y), 0)\in (\cK_n/\cR_n)\rtimes \Gamma_n=\Gamma_{n+1}$ for each $y\in \cA_n$ and $p_n\circ\phi_{x_0,x_1,\ldots, x_n}=\phi_{x_0,x_1,\ldots,x_{n-1}}$. Also, we obtain a coefficient system \[
\Z\pi_1(M)\xrightarrow{\phi_{x_0,x_1,\ldots,x_n}} \Z\Gamma_{n+1}\to \cR_{n+1},
\] 
the corresponding \emph{$(n+1)$st order Alexander module} 
\[
\cA_{n+1}=\cA_{n+1}(x_0,x_1,\ldots, x_n):=H_1(M;\cR_{n+1}),
\]
and the \emph{$(n+1)$st order Blanchfield form}
\[
B\ell_{n+1}\colon \cA_{n+1}\times \cA_{n+1}\to \cK_{n+1}/\cR_{n+1}.
\]

\subsection{Vanishing of $\rho$-invariants of order $n$}\label{subsection:vanishing-of-rho-invariant}
In this subsection, we define when a knot $K$ is said to have vanishing $\rho$-invariants of order $n$. An $\cR_n$-submodule $P$ of $\cA_n$ is said to be \emph{self-annihilating} if $P=P^\perp$ where 
\[
P^\perp:=\{y\in \cA_n\,\mid \, B\ell_n(x,y)=0\textrm{ for all }x\in P\}.
\]

The following definition is due to \cite[Theorem~4.6]{Cochran-Orr-Teichner:2003-1} (also see \cite[Definition~2.8]{Kim-Kim:2014-1}).
\begin{definition}\label{definition:vanishing-rho-invariant}
Let $n\ge 0$ be an integer. A knot $K$ has \emph{vanishing $\rho$-invariants of order $n$} if it satisfies the following conditions $(0)$--$(n)$. Let $M:=M(K)$.
\begin{itemize}
\item[($0$)] $\rho(M,\epsilon)=0$ where $\epsilon\colon \pi_1(M)\to \Gamma_0$ is the abelianization.
\item[($1$)] There exists a self-annihilating submodule $P_0$ of $\cA_0$ such that for each $x_0\in P_0$ and the corresponding representation $\phi_{x_0}\colon \pi_1(M)\to \Gamma_1$, we have $\rho(M,\phi_{x_0})=0$.

\item[($2$)] Let $x_0\in P_0$ and $\cA_1$ the corresponding first order Alexander module $\cA_1(x_0)$. Then there exists a self-annihilating submodule $P_1=P_1(x_0)$ of $\cA_1$ such that for each $x_1\in P_1$ and the corresponding representation $\phi_{x_0, x_1}\colon \pi_1(M)\to \Gamma_2$, we have $\rho(M,\phi_{x_0,x_1})=0$.

\centerline{\vdots}

\item[($n$)] Let $x_{n-2}\in P_{n-2}(x_0,x_1,\ldots,x_{n-3})$ and $\cA_{n-1}$ the corresponding $(n-1)$th order Alexander module $\cA_{n-1}(x_0,x_1,\ldots, x_{n-2})$. Then there exists a self-annihilating submodule $P_{n-1}=P_{n-1}(x_0,x_1,\ldots, x_{n-2})$ of $\cA_{n-1}$ such that for each $x_{n-1}\in P_{n-1}$ and the corresponding representation $\phi_{x_0, x_1\ldots,x_{n-1}}\colon \pi_1(M)\to \Gamma_n$, we have 
\[
\rho(M,\phi_{x_0,x_1,\ldots,x_{n-1}})=0.
\]
\end{itemize}
\end{definition}

\begin{theorem}[{\cite[Theorem~4.6]{Cochran-Orr-Teichner:2003-1}}]\label{theorem:vanishing-rho-invariant}
Let $n$ be a nonnegative integer. If a knot $K$ is $n.5$-solvable, then it has vanishing $\rho$-invariants of order $n$. In particular, if $K$ is slice, then it has vanishing $\rho$-invariants of order $n$.
\end{theorem}

\section{Proof of Theorem~\ref{theorem:main}}\label{section:proof-of-main-theorem}
In this section, we give a proof of Theorem~\ref{theorem:main}. Let $n\ge 0$ be an integer. Suppose two knots $K$ and $J$ have coprime Alexander polynomials and $K\# J$ is $n.5$-solvable. We will show that $K$ has vanishing $\rho$-invariants of order $n$ using induction on $n$ (see Definition~\ref{definition:vanishing-rho-invariant}).

Suppose $n=0$, that is, suppose $K\# J$ is 0.5-solvable. Then $K\# J$ is algebraically slice (see \cite[Remark~1.3]{Cochran-Orr-Teichner:2003-1}). Since $K$ and $J$ have coprime Alexander polynomials, it follows that $K$ and $J$ are algebraically slice (see \cite{Levine:1969-2}). Therefore $\int_{\omega\in S^1} \sigma_K(\omega)\,\,d\omega =0$ where $\sigma_K$ is the Levine-Tristram signature function of $K$. Since 
\[
\rho(M(K),\epsilon)= \int_{\omega\in S^1} \sigma_K(\omega)\,\,d\omega
\]
for the abelianization map $\epsilon: \pi_1(M(K))\to \Z$ by (2.4) in \cite[p.\ 108]{Cochran-Orr-Teichner:2004-1}, this implies that $K$ has vanishing $\rho$-invariants of order 0.

Suppose $n\ge 1$. Since $K$ is $n$-solvable, it is $(n-1).5$-solvable. It follows that $K$ has vanishing $\rho$-invariants of order $n-1$ by Theorem~\ref{theorem:vanishing-rho-invariant}. Therefore  $K$ satisfies the conditions (0)--$(n-1)$ in Definition~\ref{definition:vanishing-rho-invariant}, and it suffices to show that $K$ satisfies the condition $(n)$ in Definition~\ref{definition:vanishing-rho-invariant}.

Let $L:=K\# J$. Since $K$ and $L$ are $n$-solvable and $\cF_n$ is a group, $J$ is also $n$-solvable. Let $V$ be an $n$-solution for $J$ and $W$ be an $n.5$-solution for $L$. We construct a 4-manifold $X$ whose boundary is $M(K)$ as follows. Let $C$ be the standard cobordism with the top boundary $\partial_+C=M(K)\sqcup M(J)$ and bottom boundary $\partial_- C=-M(L)$ obtained by attaching a 1-handle and a 2-handle to $(M(K)\sqcup M(J))\times \{1\}\subset (M(K)\sqcup M(J))\times [0,1]$ and then turning it upside down (see \cite[p.\ 113]{Cochran-Orr-Teichner:2004-1}). To be precise, the 1-handle is attached so that the resulting manifold is connected, and the 2-handle is attached so that the meridians of $K$ and $J$ are equated in $\pi_1(C)$. Now we define 
\[
X:=W\cup_{M(L)}C\cup_{M(J)}V.
\]
Then, $X$ is a 4-manifold with $\partial X= M(K)$. One can easily see that $X$ is an $n$-solution using Mayer-Vietoris sequences. Recall that $\cR_0=\Q[t^{\pm 1}]$. Let 
\[
P_0:=\Ker\{i_*\colon H_1(M(K);\cR_0)\to H_1(X;\cR_0)\},
\]
where $i_*$ is the homomorphism induced from inclusion. Then, $P_0$ is a self-annihilating submodule of $\cA_0=H_1(M(K);\cR_0)$ (see \cite[Theorem~4.4]{Cochran-Orr-Teichner:2003-1}). 

Choose an element $x_0\in P_0$, then we obtain a homomorphism $\phi_{x_0}\colon \pi_1(M(K))\to \Gamma_1$ as explained in Subsection~\ref{subsection:representations-to-Gamma}. Since $X$ is an \emph{a fortiori} 1-solution, $\phi_{x_0}$ extends to $\pi_1(X)$ (see \cite[Theorem~3.6]{Cochran-Orr-Teichner:2003-1}). By abuse of notation, we denote restrictions of $\phi_{x_0}$ to subspaces of $X$ by $\phi_{x_0}$ as well. The following lemma is a key lemma of the proof. (Compare the proof of Lemma~\ref{lemma:trivial-on-V} with \cite[Definition~3.1]{Cha:2019-1}.)

\begin{lemma}\label{lemma:trivial-on-V}
	There exists an extension $\phi_{x_0}\colon \pi_1(X)\to \Gamma_1$ such that $\phi_{x_0}$ restricted to $V$ is trivial on $\pi_1(V)^{(1)}$ (hence $\phi_{x_0}\colon \pi_1(V)\to \Gamma_1$ is the abelianization onto $\Z$).
\end{lemma}
\begin{proof}
Let $\phi_{x_0}\colon \pi_1(X)\to \Gamma_1$ be any extension of $\phi_{x_0}\colon \pi_1(M(K))\to \Gamma_1$. For brevity, let $\phi:=\phi_{x_0}$. Since $\Gamma_1^{(2)}=\{e\}$, the homomorphism $\phi$ factors through $\pi_1(X)/\pi_1(X)^{(2)}$. Note that 
\[
\pi_1(X)/\pi_1(X)^{(2)}\cong (\pi_1(X)^{(1)}/\pi_1(X)^{(2)})\rtimes \Z \cong H_1(X;\Z[t^{\pm 1}])\rtimes \Z.
\]
Moreover, since $\Gamma_1$ is a $\Q$-module, $\phi$ factors through $H_1(X;\Q[t^{\pm 1}])\rtimes \Z$, that is, $\phi=\phi'\circ q$ for some homomorphisms $q\colon \pi_1(X)\to H_1(X;\Q[t^{\pm 1}])\rtimes \Z$ and $\phi'\colon  H_1(X;\Q[t^{\pm 1}])\rtimes \Z\to \Gamma_1$. 

We can decompose $H_1(X;\Q[t^{\pm 1}])\cong A\oplus B$ where 
\[
A\cong \bigoplus_{p(t)}\left(p(t)\textrm{-primary part of }H_1(X;\Q[t^{\pm 1}])\right)
\]
where $p(t)$ runs over irreducible factors of $\Delta_K(t)$. Let $p_A\colon  H_1(X;\Q[t^{\pm 1}])\to A$ be the projection homomorphism. 

Now let $Z:=W\cup C$ and $j_*\colon \pi_1(Z)\to \pi_1(X)$ be the homomorphism induced from inclusion. Define $\psi\colon \pi_1(Z)\to \Gamma_1$ to be the composition
\[
\pi_1(Z)\xrightarrow{j_*}\pi_1(X)\xrightarrow{q}  H_1(X;\Q[t^{\pm 1}])\rtimes \Z \xrightarrow{(p_A,id)} A\rtimes \Z \xrightarrow{\phi'}\Gamma_1.
\]
Then, for $\partial Z=M(K)\sqcup M(J)$, since the image of $\pi_1(M(K))$ under $q$ is contained in $A\rtimes \Z$, the map $\psi$ is equal to $\phi$ on $\pi_1(M(K))$, hence an extension of $\phi=\phi_{x_0}\colon \pi_1(M(K))\to \Gamma_1$.  Also, $\psi$ is equal to the abelianization on $\pi_1(M(J))$.  For, the image of $\pi_1(M(J))$ under $q$ is contained in $B$ since every element of $H_1(M(J);\Q[t^{\pm 1}])$ is an $\Delta_J(t)$-torsion element and $\Delta_J(t)$ is coprime to $\Delta_K(t)$.

We extend $\psi$ to $\pi_1(X)$. Note that  by Seifert-van Kampen Theorem, $\pi_1(X)\cong \pi_1(Z)* \pi_1(V)$ with amalgamation of $\pi_1(M(J))$. Note that since $V$ is an $n$-solution for $J$, we have $H_1(M(J))\cong H_1(V)$. Since $\psi$ on $\pi_1(M(J))$ is the abelianization, if we define $\psi\colon \pi_1(V)\to \Gamma_1$ by $\psi(a)=(0,\epsilon(a))\in \Q(t)/\Q[t^{\pm 1}]\rtimes \Z$ where $\epsilon\colon \pi_1(V)\to \Z$ is the abelianization, then we obtain a well-defined extension $\psi\colon \pi_1(X)\to \Gamma_1$, which satisfies the desired conditions. 
\end{proof}

By Lemma~\ref{lemma:trivial-on-V}, we may assume that $\phi_{x_0}$ is trivial on $\pi_1(V)^{(1)}$ with any choice of $x_0\in P_0$. Since $K$ is $n$-solvable via $X$, by \cite[Theorem~4.6]{Cochran-Orr-Teichner:2003-1}, the knot $K$ satisfies the conditions $(n)$ in Definition~\ref{definition:vanishing-rho-invariant}, possibly only except that $\rho(M(K),\phi_{x_0,x_1,\ldots,x_{n-1}})=0$ for $\phi_{x_0,x_1,\ldots,x_{n-1}}\colon \pi_1(M)\to \Gamma_n$.

For brevity, we let $\phi:=\phi_{x_0,x_1,\ldots,x_{n-1}}$. Now it suffices to show that $\rho(M(K),\phi)=0$.  As we have stated in Subsection~\ref{subsection:rho-invariant},
\[
\rho(M(K),\phi)=\bar{\sigma}_{\Gamma_n}(X)=\sign_G^{(2)}(X)-\sign(X).
\]
By Novikov additivity, 
\[
\bar{\sigma}_{\Gamma_n}(X)= \bar{\sigma}_{\Gamma_n}(W) + \bar{\sigma}_{\Gamma_n}(C) + \bar{\sigma}_{\Gamma_n}(V).
\]
We will show that each of the terms on the right hand side equals zero, which will complete the proof.

Firstly, $\bar{\sigma}_{\Gamma_n}(W)=0$. For, $W$ is an $n.5$-solution and $\Gamma_n^{(n+1)}=\{e\}$ (see \cite[Theorem~4.2]{Cochran-Orr-Teichner:2003-1}). Secondly, $ \bar{\sigma}_{\Gamma_n}(C)=0$ since $\sign(C)=\sign_{\Gamma_n}^{(2)}(C)=0$: one can see $\sign(C)=0$ since $H_2(\partial C)\to H_2(C)$ is surjective. Also, $\sign_{\Gamma_n}^{(2)}(C)=0$ due to the proof of \cite[Lemma~4.2]{Cochran-Orr-Teichner:2004-1}.

Finally, it only remains to show $\bar{\sigma}_{\Gamma_n}(V)=0$. Since $V$ is an $n$-solution and $n\ge 1$, we have $\sign(V)=0$. We show that $\sign_{\Gamma_n}^{(2)}(V)=0$. Since $V$ is an $n$-solution, it suffices to show that $\phi\colon \pi_1(V)\to \Gamma_n$ factors through $\pi_1(V)/\pi_1(V)^{(n)}$ due to \cite[Theorem~4.2]{Cochran-Orr-Teichner:2003-1} and its proof. For, if $\phi$ on $\pi_1(V)$ factors through $\pi_1(V)/\pi_1(V)^{(n)}$, then $H_2(V;\cK_n)$ have a half-rank summand on which the intersection form with $\cK_n$ coefficients vanishes, and this will imply  that $\bar{\sigma}_{\Gamma_n}(V)=0$. (Recall that $\cK_n$ is the skew quotient field of $\Z\Gamma_n$.)

Now we show that $\phi\colon \pi_1(V)\to \Gamma_n$ factors through $\pi_1(V)/\pi_1(V)^{(n)}$. Recall that there are canonical projections $p_{i-1}\colon\Gamma_i=(\cK_{i-1}/\cR_{i-1})\rtimes \Gamma_{i-1} \to \Gamma_{i-1}$. Composing the maps $p_i$ for $1\le i\le n-1$, we obtain a canonical projection $p\colon \Gamma_n\to \Gamma_1$ such that  $p\circ \phi = \phi_{x_0}$. By Lemma~\ref{lemma:trivial-on-V}, the map $\phi_{x_0}\colon \pi_1(X)\to \Gamma_1$ maps $\pi_1(V)^{(1)}$ to $\{e\}$. Therefore, 
\[
\Im\{\pi_1(V)^{(1)}\xrightarrow{\phi} \Gamma_n \xrightarrow{p} \Gamma_1\} = \{e\}.
\] 
Therefore, letting $G$ be a subgroup of $\Gamma_n$ such that 
\[
G:= \cK_{n-1}/\cR_{n-1}\rtimes \left(\cK_{n-2}/\cR_{n-2} \rtimes \left(\cdots \left(\cK_1/\cR_1\rtimes \{0\})\rtimes \{0\})\right)\cdots\right)\right),
\]
we have
\[
\Im\{\pi_1(V)^{(1)}\xrightarrow{\phi} \Gamma_n\} \subset G.
\]
Note that if $H$ is a group such that $H=H_1\rtimes H_2 $, then $H^{(1)}\subset H_1\rtimes H_2^{(1)}$. Using this iteratively and noting that $\cK_i/\cR_i$ are abelian groups for all $i$, one can see that $G^{(n-1)}=\{e\}$. Therefore,
\[
\Im\{\pi_1(V)^{(n)}\xrightarrow{\phi}\Gamma_n\}= \Im\{(\pi_1(V)^{(1)})^{(n-1)}\xrightarrow{\phi}\Gamma_n\} \subset G^{(n-1)}=\{e\},
\]
and hence $\phi$ factors through $\pi_1(V)/\pi_1(V)^{(n)}$. This completes the proof of Theorem~\ref{theorem:main}.

\providecommand{\bysame}{\leavevmode\hbox to3em{\hrulefill}\thinspace}
\providecommand{\MR}{\relax\ifhmode\unskip\space\fi MR }
\providecommand{\MRhref}[2]{%
  \href{http://www.ams.org/mathscinet-getitem?mr=#1}{#2}
}
\providecommand{\href}[2]{#2}

\end{document}